\documentclass[11pt]{article}

\usepackage{amsmath,amsthm,amsfonts,amssymb,color, graphicx, epsfig}
\usepackage{srcltx}
\usepackage{soul}
\usepackage[margin=1.25in]{geometry}
\usepackage{hyperref}

\newtheorem{teo}{\bf Theorem}
\newtheorem*{teo*}{\bf Theorem}
\newtheorem{theo}{\bf Theorem}
\newtheorem{lema}{\bf Lemma}[section]
\newtheorem{cor}[lema]{\bf Corollary}
\newtheorem{propo}[lema]{\bf Proposition}
\newtheorem{defi}[lema]{\bf Definition}

\newtheorem{prob}{\bf Problem}
\theoremstyle{definition}

\newtheorem{remas}[lema]{\bf Remarks}

\newtheorem{example}[lema]{\bf Example}

\newcommand{\NN}{\mathbb{N}}       
\newcommand{\RR}{\mathbb{R}}       
\newcommand{\CC}{\mathbb{C}}       

\renewcommand{\phi}{\varphi}
\renewcommand{\leq}{\leqslant}
\renewcommand{\geq}{\geqslant}

\newcommand{\diff}{\mathrm{Diff}}

\newcommand{\cA}{{\mathcal{A}}}

\newcommand{\diffdiss}{\diff_{\operatorname{SDiss}}}

  \def\CC{{\mathbb C}} \def\DD{{\mathbb D}}

 \def\NN{{\mathbb N}}  
 \def\RR{{\mathbb R}}

\title{Strongly dissipative surface diffeomorphisms}
\author{Sylvain Crovisier\footnote{S.C. is partially supported by the ERC project 692925 -- NUHGD.}, \; Enrique  Pujals\footnote{E.R.P. is partialy supported by  CNPq.}}
\date{\today}
\begin{document}

\maketitle

\begin{abstract}
We introduce a class of volume-contracting surface diffeomorphisms whose
dynamics is intermediate between one-dimensional dynamics and general surface
dynamics. For that type of systems one can associate
to the dynamics a reduced one-dimensional model and it is proved a type of
$C^\infty$-closing lemma on the support of every ergodic measure.
We also show that this class contains H\'enon maps {\color{black} with Jacobian in $(-1/4,1/4)$.}
 \end{abstract}

\section{ Strong dissipation}

We would like to highlight two main topics in dynamical systems:
the $C^\infty$-closing lemma and the search of reduced dynamical models that encapsulates the main properties of an open class of systems. 

The first one was posed by Poincar\'e and refers to the problem of finding periodic orbits nearby recurrent points either for a system or a $C^r$-perturbation. It was solved in \cite{Pu} for the $C^1$-category. In higher topology it remains widely open except for certain particular classes of maps: 
rational maps of the Riemann sphere~\cite{F,J}, one-dimensional real endomorphisms~\cite{LSY}, and recently Hamiltonian
surface diffeomorphisms~\cite{AI}. For general surfaces diffeomorphisms, it remains completely open. 

The second topic is present in the whole theory of dynamics and consists in looking for simplified models that could extract the main features of systems. That approach goes from finding discrete topological representations using symbolic dynamics, first return maps for continuous dynamical systems and reducing the dimension of the space. 

Related to that problem, Poincar\'e realized  that flows may be reduced
to discrete systems using one-codimensional sections.
For instance,  the H\'enon type map $(x, y)\mapsto (1-ax^2 + y, bx)$
appears naturally~\cite{henon} as  the first return map of some three-dimensional
flows. In case the flow is dissipative, the corresponding invertible system is area
contracting, which heuristically means that the essential dynamics is confined to a one-dimensional subspace; so the initial dynamics may
share some essential features of the dynamics of the interval map $f(x)=1-ax^2.$
In practice, the two-dimensional systems are much more difficult to describe:
much less is known for the H\'enon map than for the quadratic family. Moreover,
many features that are not detected by the one-dimensional reduction have no one-dimensional counterpart:
for instance there exists a residual set of H\'enon maps exhibiting infinitely many periodic attractors with unbounded period (Newhouse phenomenon), but generic smooth one-dimensional maps have an upper bound on the period of the attracting periodic points.

Here, we introduce a new class of dissipative surface diffeomorphisms that we call {\it strongly dissipative diffeomorphisms}  that captures certain properties of one dimensional map but keeps two-dimensional features showing all the well known complexity of dissipative surface diffeomorphisms. The dynamics of the new class, in some sense, is intermediate between one-dimensional dynamics and general surface diffeomorphism. Moreover, under some
hypothesis on the relation of the Jacobian with the $C^1$-norm and the oscillation of the Jacobian in the attractor, it is proved that the strong dissipativeness is an open property (see theorem \ref{p.strong}) and it is satisfied by diffeomorphisms close to one dimensional endomorphisms, proving in particular, that the new class is also non empty (see theorem \ref{t.1D-bis}).
The class of strongly dissipative surface diffeomorphisms may be compared to the class of moderately dissipative complex H\'enon maps
considered in~\cite{LP}: using different tools, an upper bound on the Jacobian is used in order to control stable manifolds.

The theory of  real one-dimensional dynamics is leveraged on the order structure of the interval, a feature that does not exist for the plane. However, under dissipativeness, almost every point of any ergodic measure has a stable manifold that could help ``to order the trajectories''; in fact, the strong dissipativeness hypothesis, which is nothing else than assuming that the stable manifolds separate an attracting domain (see definition \ref{SD defi}), helps to recover in the particular case of the disk a partial order and to induce a rich one-dimensional structure.   Using that simple observation, the dynamics can be reduced to a continuous non-invertible map acting on an ordered one-dimensional path connected metric space (see theorem \ref{reduction}). One can hope that this result could  lead to obtain many others that hold for one-dimensional systems.
Note that for general surface homeomorphism, another reduction has been developed by Le Calvez~\cite{L}: it provides a foliation transverse to the
dynamics on the complement to maximal sets of fixed points; in this case the foliation is in general not invariant and the leaves are not proper.

A clear result that highlights the richness of the strongly dissipative class is our last theorem (see theorem \ref{t.measure revisited}) that shows that  the periodic points are dense in the support of any invariant measure (in particular, for the H\'enon type maps, see corollary \ref{c.henon}); in that sense, we get a $C^\infty$-closing lemma (without perturbing) for invariant measures. This was proved in \cite{K} for hyperbolic measures, however that result does not apply  to measures that have a zero Lyapunov exponent, which is the case for instance, of surfaces diffeomorphisms with zero entropy and in particular the ones in the boundary of chaos (see \cite{LM}). An strong application of that result is obtained in ~\cite{CPT} where the dynamics of  strongly dissipative diffeomorphisms of the disk with zero entropy is studied. 
In short, in the last theorem, we conclude that for strongly dissipative diffeomorphisms of the disk, the closure of all periodic points contains
the closure of the union of the supports of all invariant probability measures~\footnote{It is worth to mention that based on the above stated result, for  strongly dissipative diffeomorphisms of the disk with zero entropy, in \cite{CPT} is proved that the periodic points are dense in the non-wandering set.}. To conclude a general closing lemma, it would be needed to prove that generically the recurrent set is contained in that set.

 \bigskip
 
Let us consider a boundaryless surface $S$ and a $C^{r}$-diffeomorphism
$f\colon S\to f(S)\subset S$, where $r>1$.
If $f$ is \emph{dissipative}, i.e. if $|\det(Df(x))|<1$ for any $x\in S$,
then any $f$-invariant ergodic probability measure $\mu$
which is not supported on a hyperbolic sink
has one negative Lyapunov exponent
and another one which is non-negative.
In particular
for $\mu$-almost every point $x$, there exists a well-defined one-dimensional
stable manifold $W^s(x)$.
{\color{black} We denote by $W^s_{S}(x)$ the connected component of $W^s(x)$ which contains $x$.}

We introduce a class of surface diffeomorphisms which strengthen the notion of dissipation:

\begin{defi}\label{SD defi}
A $C^r$-diffeomorphism $f\colon S\to f(S)\subset S$ is \emph{strongly dissipative} if
\begin{itemize}
\item[--] $f(S)$ is contained in a compact subset of $S$,
\item[--] $f$ is dissipative,
\item[--] for any ergodic measure $\mu$ which is not supported on a hyperbolic sink, and for
$\mu$-almost every point $x$,
{\color{black} each of the two connected components of $W^s_S(x)\setminus \{x\}$ meets $S\setminus f(S)$.}
\end{itemize}
{\color{black} \rm (When $S$ is the disc $\DD$, the last condition
says that $W^s_\DD(x)$ separates the disc, see figure~\ref{f.strong-dissipation}.)}
\end{defi}

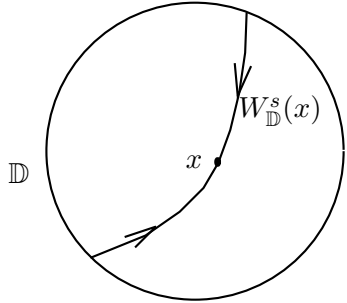
\begin{figure}
\begin{center}
\includegraphics[scale=0.5]{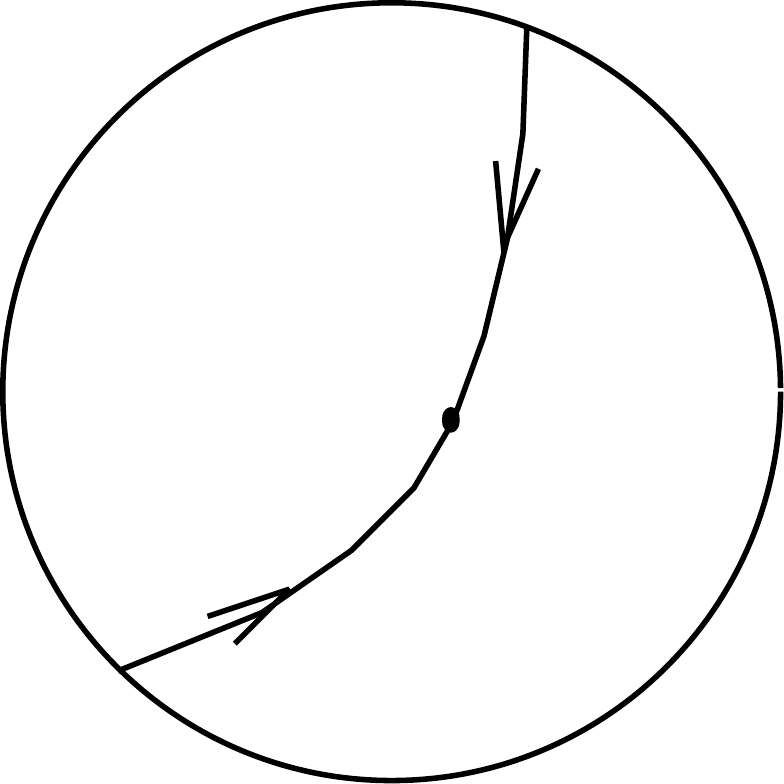}
\end{center}
\begin{picture}(0,0)
\put(217,76){$\tiny x$}
\put(237,95){$\tiny W^s_\DD(x)$}
\put(150,70){$\DD$}
\end{picture}
\caption{A separating stable manifold in the disc.
\label{f.strong-dissipation}}
\end{figure}

We denote $\diffdiss^r(S)$, $r>1$, the set of strongly dissipative $C^r$-diffeomorphisms.
\medskip

As it is shown at the end of  section~\ref{t.1D proof},
there are dissipative diffeomorphisms (meaning that the first two items in definition \ref{SD defi} are satisfied) 
which are not strongly dissipative; however, those examples are not Kupka-Smale and in particular observe that any dissipative Axiom A exhibiting  transversality is strongly dissipative. So it is natural to ask the following:

\begin{prob}
Is any Kupka-Smale dissipative diffeomorphism also strongly dissipative?
\end{prob}

When the Jacobian is {\color{black} small enough}, then the strong dissipation is an $C^1$-open property:

\begin{teo}\label{p.strong}
{\color{black} Let  $f\in \diffdiss^{2}(S)$, satisfying\footnote{\color{black} In the following we will also denote $\|Df\|:=\sup_x \|Df(x)\|$ and $|\det Df|:=\sup_x |\det Df(x)|$.
Hence condition~\eqref{e.dissipation2} can be written as $|\det Df|\leq \|Df^{-1}\|^{-9/10}$.}
for $x,y\in S$ and any unit vector $u\in T_yS$,
\begin{equation}\label{e.dissipation2}
|\det Df(x)|\leq  \|Df(y).u\|^{9/10}.
\end{equation}
Then, inside any $C^2$-bounded set of $C^2$-diffeomorphisms,
any diffeomorphism $g$ that is $C^1$-close to $f$ is strongly dissipative.}
\end{teo}

This result is based on a kind of continuity of Pesin's blocks and uniformity of stable manifolds,
which will be stated and proved in section~\ref{ss.stable}.
It uses that for non periodic ergodic measures, the negative Lyapunov exponent
is close to (the logarithm of) the minimum contraction of $f$. This property is reminiscent of
the works on the non-uniform hyperbolicity of some H\'enon maps~\cite{BC}.
\medskip

We will show that the strong dissipation is satisfied by dynamics
close to one-dimensional endomorphisms.
{\color{black} That result is stated in its full generality in section~\ref{t.1D proof}, Theorem~\ref{t.1D}.
In the case of polynomial automorphisms, one can also apply tools from complex analysis and in particular Wiman theorem,
as it was explained to us by Dujardin and Lyubich. This has already be used in~\cite{LP,DL}
and a small variant of the arguments there gives the following stronger form of Theorem~\ref{t.1D} for the H\'enon map:

\newcounter{theothird}
\setcounter{theothird}{\value{teo}}
\begin{teo}[Version for the H\'enon map] \label{t.1D-bis} For any $a\in (1,2)$ and for $b\in (-\frac{1} 4,\frac 1 4)\setminus \{0\}$, the H\'enon map
\begin{equation}\label{e.henon} H_{a,b}\colon \;(x,y)\mapsto (1-ax^2+y,-bx)
\end{equation}
is strongly dissipative on the surface $S=\{(x,y):\; |x|<1/2+1/a, \; |y|<1/2-a/4\}$.
\end{teo}}

The following shows that, conversely to the previous theorem, strongly dissipative diffeomorphisms (in particular H\'enon maps with $|b|$ small)
have a one-dimensional structure. This somewhat generalizes Williams construction~\cite{W} of branched
manifolds associated to hyperbolic surface attractors: in  that case, the reduction obtained through the quotient of the stable manifolds gives an endomorphism on a branched one-dimensional manifold. For strongly dissipative theorem, even lacking a uniformly stable foliation, using the strong dissipation, a reduced one-dimensional dynamics is obtained as a continuous non-invertible map acting on a real tree.

We recall that a real tree is a path connected metric space such that for any two points $a,b$ there exists a unique
subset homeomorphic to $[0,1]$ whose endpoints are $a$ and $b$.

\begin{teo}\label{reduction}
Let $f\in \diffdiss^{r}(\DD)$, $r\!>\!1$, be a strongly dissipative diffeomorphism of the disc.
Then there exists a semi-conjugacy
$\pi\colon (\mathbb{D},f)\to (X,h)$
to a continuous map $h$ on a compact real tree which induces an injective map on
the set of non-atomic ergodic measures $\mu$ of $f$.
{\color{black} Moreover the entropies of $\mu$ and $\pi_*(\mu)$ are the same.}
\end{teo}

As an application of the notion of strong dissipation,
we show that the periodic points approximate the support of any invariant measure,
generalizing the result for one-dimensional endomorphisms. Our argument also provides a simpler proof of the result presented in \cite{LSY} (see the section~\ref{sec measures}) and also proving that periodic points are dense in the non-wandering set of endomorphisms of compact real tree as defined above.

\begin{teo}\label{t.measure revisited}
For $f\in \diffdiss^{r}(\DD)$, $r>1$, the support of any $f$-invariant probability $\mu$ is contained in the closure of the periodic points.
In particular if $f$ preserves a non-atomic ergodic measure, then there are infinitely many periodic points with unbounded period.
\end{teo}

As a consequence we obtain:

\begin{cor}\label{c.henon} \color{black}
For any $a\in (1,2)$ and $b\in (-\frac{1} 4,\frac 1 4)\setminus \{0\}$,
the support of any probability measure $\mu$
which is invariant by the H\'enon map $H_{a,b}$
is contained in the closure of the periodic points.
\end{cor}
{\color{black} Dujardin has recently obtained (with a completely different approach) a similar statement for the complex H\'enon automorphisms,
once the Jacobian $b$ satisfies $0<|b|<1$: his result provides a dense set of periodic points in the union of the support of the ergodic probability measures
in $\CC^2$ which are not supported on a periodic circle. Note that in the case of real H\'enon maps, it does not conclude about the existence of periodic points in the real plane
and does not imply the Corollary~\ref{c.henon} above.}

\paragraph{Acknowledgements.}
This work started during the preparation of~\cite{CKKP}, where a dichotomy for strong dissipative of the annulus is proved,   and we are indebted to
Alejandro Kocsard and Andres Koropecki for the discussions we exchanged on this topic.
{\color{black} We also thank Romain Dujardin and Mikhail Lyubich for their explanations about Wiman theorem:
the stronger version of Theorem~\ref{t.1D} for the H\'enon family stated in the introduction is due to them.
Finally, we thank the referee for his comments which improved the first version of this text.}


\section{Stable manifolds}\label{ss.stable}
The proof of theorems \ref{p.strong} and \ref{t.1D} requires a strong version of the stable manifold theorem
for non-uniformly hyperbolic orbits.
{\color{black} As it is explained in the remarks below, it provided a uniformity of Pesin's stable manifold theorem
with respect to the measure. The assumptions not only require a contraction and a domination with the transverse direction,
but also a pinching.}

\begin{teo}[Stable manifold at non-uniformly hyperbolic points]\label{t.stable}
{\color{black} Consider a compact set $\Lambda\!\subset\! S$, two neighborhoods $U,V$,
a $C^2$-bounded set $\mathcal{D}$ of $C^2$-diffeomorphisms $f\colon U\to V$
and $\sigma,\tilde \sigma,\rho,\tilde \rho\in (0,1)$
such that $\frac{\tilde \sigma\tilde \rho}{\sigma\rho}>\sigma$.

Then, for any $f\in \mathcal D^2$, the points $x\in \cap_{n\geq 0} f^{-n}(\Lambda)$ having a direction $E\subset T_xS$
satisfying
\begin{equation}\label{e.stable}
\forall n\geq 0,\;\;\; {\tilde \sigma}^{n}\leq \|Df^n(x)_{|E}\|\leq  \sigma^{n},\;
\text{ and }\; {\tilde \rho}^n\leq \frac{\|Df^n(x)_{|E}\|^2}{|\det Df^n(x)|}\leq \rho^{n},
\end{equation}
have a one-dimensional stable manifold
varying continuously in the $C^1$-topology with the point $x$
and with the diffeomorphism $f$ in the space $\mathcal D$ endowed with the $C^1$-topology.}
\end{teo}

\begin{remas}
\begin{enumerate}
\item {\color{black} The statement is also valid for $C^{1+\alpha}$-diffeomorphisms $\alpha\in (0,1)$
if $\mathcal{D}$ is a $C^{1+\alpha}$-bounded set of $C^{1+\alpha}$-diffeomorphisms
and if the condition $\frac{\tilde \sigma\tilde \rho}{\sigma\rho}>\sigma^\alpha$ holds.}

\item The second part of condition~\eqref{e.stable} can be restated (see also~\cite{PR}):
$E$ repels exponentially for the action on the directions.
More precisely,
$Df$ induces an action on the unit tangent bundle; if $\gamma_x(E)$ denotes its derivative
at $(x,E)$ along the circle $T^1_xS$, then $|\gamma^n_x(E)|\geq \rho^{-n}$ for each $n\geq 0$.

\item This result implies the stable manifold theorem of Pesin theory
for surface diffeomorphism.
{\color{black} Indeed if $\mu$ is an ergodic measure having some Lyapunov exponents
$\lambda^-<\lambda^+$ with $\lambda^-<0$,
then for any $\varepsilon>0$, there exists $N\geq 1$
and a set with $\mu$-measure larger than $1-\varepsilon$ of points $x$
such that condition~\eqref{e.stable} holds for $f^N$
with $\tilde \sigma,\sigma=\exp(N(\lambda_-\pm \varepsilon))$
and $\tilde \rho,\rho=\exp(N(\lambda_--\lambda_+\pm\varepsilon))$.}
 
In particular this result also
gives a uniformity with respect to the measure.
\item This result is close to~\cite[section 5]{BC} where stable manifolds are obtained
by successive approximations.
\end{enumerate}
\end{remas}

\begin{proof}
{\color{black} As in the proof of Pesin's stable manifold, the idea is to consider a sequence of local charts
following the orbit of $x$. These charts are scaled so that the dynamics of $f$ behaves as a uniformly hyperbolic
diffeomorphism. On the other hand, their size has to be controlled in order to guarantee that the stable
manifold in the charts corresponds to a stable manifold for the initial diffeomorphism.}
\medskip

{\color{black} Since $\frac{\tilde \sigma\tilde \rho}{\sigma\rho}>\sigma$,
one can choose $\lambda_1\in (\sigma,1)$ and
$\lambda_2\in (0,\tilde \rho)$ satisfying
\begin{equation}\label{e.condition}
\frac{\tilde \sigma\lambda_2}{\lambda_1\rho}>\lambda_1.
\end{equation}

Let $C_0>0$ such that
\begin{equation}\label{e.bound-C_0}
C_0>\sum_{k\geq 0} (\sigma/\lambda_1)^k \text{ and }
C_0>\sum_{k\geq 0} (\lambda_2/\tilde \rho)^k.
\end{equation}}

The dynamics of $f$ in a neighborhood of the forward orbit of $x$ can be lifted
by the exponential map to
the tangent bundle as maps $g_n\colon T_{f^n(x)}S\to T_{f^{n+1}(x)}S$
defined on uniform neighborhoods of $0$. Let $F:=E^\perp$.
One considers the orthogonal decomposition
$T_{f^n(x)}S=E_n\oplus F_n$ such that $E_n:=Df^n(E)$ and $F_n=E_n^\perp$.
We then set
$$m_n=\|Df^n_{|E}(x)\| \text{ and }
M_n=|\det(Df^n(x))|/m_n.$$
In these coordinates, the map $Dg_n(0)$ has the form
$\begin{pmatrix} m_{n+1}/m_n & *\\ 0 &M_{n+1}/M_n\end{pmatrix}$.

Note that
\begin{equation}\label{e.bound-m}
\|Df^{-1}\|^{-1}\leq \frac{m_{k+1}}{m_k}\leq \|Df\|,\quad
{M_{k+1}/M_k}\leq \|Df^{-1}\| \;|\det Df|.
\end{equation}
\medskip

We introduce the linear change of coordinates $\Delta_n$ on $T_{f^n(x)}S$
which is defined in the coordinates $E_n\oplus F_n$ by
the diagonal map $\Delta_n=\operatorname{Diag}(A_n,A_n\;B_n)$ where 
$$A_n=\sum_{k\geq 0}\lambda_1^{-k}m_{n+k}/m_n$$
$$B_n=\sum_{k= 0}^n\lambda_2^{k-n}\frac{M_k/M_n}{m_k/m_n}.$$
Assumptions~\eqref{e.stable} and~\eqref{e.bound-C_0} imply that $A_n$ is finite and $A_0\leq C_0$.
Note that $A_n,B_n$ are larger than or equal to $1$ so that
$\|\Delta_n\|=A_nB_n$ and $\|\Delta_n^{-1}\|=A_n^{-1}<1$.
An easy computation gives:
\begin{equation}\label{e.inductionA}
A_{n+1}\;\frac{m_{n+1}}{m_n}\;A_{n}^{-1}=\lambda_1(1-A_n^{-1})<\lambda_1,
\end{equation}
\begin{equation}\label{e.inductionB}
B_{n+1}\frac{M_{n+1}}{M_n}B_n^{-1}=\lambda_2^{-1}\frac{m_{n+1}}{m_n}+\frac {M_{n+1}}{M_n}B_n^{-1}>\lambda_2^{-1}\frac{m_{n+1}}{m_n}.
\end{equation}
In particular
\begin{equation}\label{e.lowerA}
\frac{A_{n+1}}{A_n}\leq \lambda_1\frac{m_{n}}{m_{n+1}}\leq \lambda_1\|Df^{-1}\|,
\end{equation}
{\color{black}
\begin{equation}\label{e.uniformA2}
A_n\leq A_0\;\lambda_1^n\frac{m_0}{m_n}
\leq C_0\;\lambda_1^n\; \tilde \sigma^{-n}.
\end{equation}}
With~\eqref{e.bound-m}, we also have
\begin{equation}\label{e.A-lower} \frac 1 {\|Df\|\; \|Df^{-1}\|} \leq \frac{A_{n+1}}{A_n}.
\end{equation}
{\color{black} From~\eqref{e.stable} and~\eqref{e.bound-C_0} we have
\begin{equation}\label{e.uniformB}
B_n= \frac{m_n^2\lambda_2^{-n}}{|\det(Df^n(x))|}
\sum_{k= 0}^n \lambda_2^k\frac{|\det Df^k(x)|}{m_k^2}\leq
\frac{m_n^2\lambda_2^{-n}}{|\det(Df^n(x))|}
\sum_{k= 0}^n (\lambda_2/\tilde \rho)^k
\leq (\rho/\lambda_2)^nC_0.
\end{equation}}

\medskip
One then defines the local diffeomorphism $h_n=\Delta_{n+1}\circ g_n\circ \Delta_n^{-1}$
and its tangent part $H_n:=\Delta_{n+1}Dg_n(0)\Delta_n^{-1}$.
The map $H_n$ has the form
$\begin{pmatrix} a & d\\ 0 &c\end{pmatrix}$.
Using~\eqref{e.bound-m}, \eqref{e.inductionA}, \eqref{e.inductionB}, \eqref{e.lowerA} and~\eqref{e.A-lower},
one gets the estimates
\begin{equation}\label{e.a}
\frac 1 {\|Df\|\; \|Df^{-1}\|^{2}} \leq |a|=A_{n+1}\frac{m_{n+1}}{m_n}A_{n}^{-1}<{\color{black}\lambda_1},
\end{equation}
\begin{equation}\label{e.c}
\begin{split} {\color{black} \lambda_2^{-1}} |a|={\color{black} \lambda_2^{-1}}A_{n+1}\frac{m_{n+1}}{m_n}A_n^{-1}&<|c|=A_{n+1}{B_{n+1}}\frac{M_{n+1}}{M_n}B_n^{-1}A_n^{-1}\\
&\quad \leq ({\color{black} \lambda_1\lambda_2^{-1}}\|Df\|\;\|Df^{-1}\|+{\color{black} \lambda_1}\|Df^{-1}\|^2\;|\det Df|)
\end{split}\end{equation}
\begin{equation}\label{e.d}
|d|\leq A_{n+1}\;\|Df\|\;A_{n}^{-1}{B_{n}^{-1}}=\|Df\|\;|a|\frac{m_n}{m_{n+1}}{B_{n}^{-1}}\leq \|Df\|\;\|Df^{-1}\|\;|a|.
\end{equation}
In particular, the $H_n,H_n^{-1}$ are uniformly bounded,
and there exists a horizontal cone which is uniformly contracted into itself 
and whose vectors are uniformly expanded under $H_n^{-1}$ by a factor larger than ${\color{black} \lambda_1^{-1}}$.
\medskip

{\color{black} Since the diffeomorphism $f$ is $C^r$, $r>1$,
there exists $\alpha>0$ and $C_f>0$
such that}
$$\|Dh_n(y)-Dh_n(0)\|\leq
\|\Delta_{n+1}\|\;\|\Delta_n^{-1}\|\;C_f \|\Delta_n^{-1}\|\;\|y\|^\alpha\leq
C_f\|\Delta_{n+1}\|\;\|y\|^\alpha.$$
Let us choose $\varepsilon>0$ small.
One can extend  $h_n$ as a global $C^1$-diffeomorphism
$\widehat h_n\colon T_{f^n(x)}S\to T_{f^{n+1}(x)}S$
which is $\varepsilon$-close to the linear map $H_n$ for the $C^1$-topology,
and which coincides with $h_n$ on the ball centered at $0$ and of radius $r_n$ such that
\begin{equation}\label{e.boundr}
{\color{black} r_n^\alpha:=\frac{\varepsilon}{C_f\;\|\Delta_{n+1}\|}=\frac{\varepsilon}{C_f\;A_n\; B_n}.}
\end{equation}
This gives from~\eqref{e.uniformA2}, \eqref{e.uniformB}
and~\eqref{e.condition}:
$${\color{black} r_n^\alpha> \frac{\varepsilon}{C_fC_0^2}\bigg(\frac {\tilde \sigma}{\lambda_1}\bigg)^{n}
\bigg(\frac {\lambda_2}{\rho}\bigg)^{n}> \frac{\varepsilon}{C_fC_0^2}\lambda_1^{n}}.$$

One has obtained a uniformly bounded family of diffeomorphisms
$(\widehat h_n)$ whose inverses expand uniformly a horizontal cone.
The classical stable manifold theorem for sequences of diffeomorphisms
(see for instance~\cite{KH}) asserts that a uniform family of $C^1$-graphs is preserved.

For $r>0$ small, the ball of radius $r$ and centered at $0$ in the graph of $T_xS$
is a curve contracted by the composition $\widehat h_{n-1}\circ\dots\circ \widehat h_0$
by  more than ${\color{black} \lambda_1^n}$, hence is contained in the ball of radius $r_n$.
Conjugating by the coordinates changes $(\Delta_n)$, this proves that
this curve (which has uniform size) is exponentially contracted by the $g_n$,
hence is a stable manifold.
{\color{black} All the constants are still valid for $C^r$-diffeomorphisms that are $C^1$-close to $f$
and have the same $C^r$-bound.}
Since the stable manifold for sequences of diffeomorphisms
depends continuously on $(\widehat h_n)$ for the $C^1$-topology,
the stable manifold for the surface diffeomorphism depends continuously on $(x,f)$.
\end{proof}

\section{Robustness: proof of theorem \ref{p.strong}}
\label{p.strong proof}
{\color{black} The robustness of the strong dissipation is obtained from the uniformity of the stable manifolds with respect to
the measure and the diffeomorphism. The proof requires to check that condition~\eqref{e.stable} holds (for any measure not supported on a sink)
on a set with uniform measure. This result is based on the following version of Pliss lemma.}

\begin{lema}[Pliss]\label{pliss}
For any $\alpha_1<\alpha_2<\alpha_3$,
and any sequence $(a_n)\in (\alpha_1,+\infty)^\NN$
satisfying
$$\limsup_{+\infty} \frac 1 n (a_0+\dots+a_{n-1})\leq \alpha_2,$$
there exists a collection of integers $0\leq n_1<n_2<\dots$ such that
\begin{itemize}
\item[--] for any $k\geq 1$ and $n>n_k$, one has
$\frac 1 {n-n_k} (a_{n_k}+\dots+a_{n-1})\leq \alpha_3$,
\item[--] the upper density $\limsup \frac {k} {n_k}$ of the sequence $(n_k)$
is larger than $\frac{\alpha_3-\alpha_2}{\alpha_3-\alpha_1}$.
\end{itemize}
\end{lema}
{\color{black} \begin{proof}[Sketch of the proof]
The argument is similar to~\cite[Chapter IV, Lemma 11.8]{M-book}.
Up to replace $a_n$ by $a_n-n\;\alpha_3$ and each $\alpha_i$ by $\alpha_i-\alpha_3$,  one can assume that $\alpha_3=0$.
We build the sequence $(n_k)$ inductively: $n_1$ realizes the maximum of the set $\{a_0+\dots+a_{n-1}, \; n\geq 0\}$, which exists
since $\limsup_{+\infty} \frac 1 n (a_0+\dots+a_{n-1})<0$. One chooses $n_{k+1}$ as the
smallest integer $n>n_k$ which realizes the maximum of $\{a_0+\dots+a_{n-1}, \; n>n_k\}$.
Hence, the integers $n_k$ satisfy for any $n>n_k$ the inequality $a_{n_k}+\dots+a_{n-1}\leq 0$ as required.

In order to estimate the density, one first notes that
$$a_0+\dots+a_{n_{k+1}-1}\geq (a_0+\dots+a_{n_{k}-1})+a_{n_k}\geq (a_0+\dots+a_{n_{k}-1})+\alpha_1.$$
One deduces inductively that for each $k$,
$$a_0+\dots+a_{n_{k}-1}\geq (k-2)\;\alpha_1\; + a_0+\dots+a_{n_2-1}.$$
Combining with our assumption, one gets
$\alpha_2\geq \limsup_k \frac k {n_k}\alpha_1$, which gives since $\alpha_1<0$:
$$\limsup_k \frac k {n_k}\geq \frac{\alpha_2}{\alpha_1}=\frac{\alpha_3-\alpha_2}{\alpha_3-\alpha_1}.$$
\end{proof}}
\medskip

{\color{black} For any $C^{1}$-diffeomorphism $f$ and $\sigma,\tilde \sigma,\rho,\tilde \rho\in (0,1)$, we introduce
$A_f(\sigma,\tilde \sigma,\rho,\tilde \rho)$,  the compact set of points $x\in S$ such that there exists a one-dimensional
subspace $E\subset T_xS$
satisfying~\eqref{e.stable}.}
The previous lemma has the following consequence.
\begin{propo}\label{p1}
{\color{black} Consider $f\in \diff^1(S)$, an invariant compact set $\Lambda$
and $$D=\sup_{x\in \Lambda} |\det D f(x)|,\quad m
=\sup_{x\in \Lambda} \|Df^{-1}(x)\|^{-1}.$$
Let $\tilde \sigma=m$,
$\tilde \rho=m^2/D$, $\sigma=D^{4/5}$ and $\rho=D^{3/4}$.

If $D<m^{9/10}$, then $\frac{\tilde \sigma\tilde \rho}{\sigma\rho}>\sigma$
and for any ergodic measure $\mu$ on $\Lambda$ which is not supported on a sink,
the measure of the set $A_f(\sigma,\tilde \sigma,\rho,\tilde \rho)$
is larger than $1/6$.}
\end{propo}
\begin{proof}
{\color{black} For any ergodic measure $\mu$ on $\Lambda$ which is not supported on a sink,
the two Lyapunov exponents $\lambda^-\leq \lambda^+$
satisfy $\lambda^-\leq \log(D)$ and $0\leq \lambda^+$.
Since $\mu$-almost every point $x$ satisfies the Oseledets theorem, for any unit vector $u\in E^s(x)$ we have
$$\frac 1 n \log \|Df^n(x).u\|\underset{n\to +\infty}\longrightarrow \lambda^-\leq \log(D),$$
$$\frac 1 n \log\frac{\|Df^n(x)|_{E^s}.u\|^2}{|\det(Df(x))|}
\underset{n\to +\infty}\longrightarrow  (\lambda^--\lambda^+)\leq \log(D).$$
On the other hand we have the bounds
$$\|Df^n(x).u\|\geq m^n,\quad \frac{\|Df^n(x)|_{E^s}.u\|^2}{|\det(Df(x))|}\geq \left(\frac
{m^2} {D}\right)^n.$$
Since $D<m^{9/10}$, we check:
$$\frac{\tilde \sigma\tilde \rho}{\sigma\rho}=m^3D^{-1-4/5-3/4}> D^{47/60}> \sigma.$$
Using Pliss lemma~\ref{pliss}, the first condition of~\eqref{e.stable}
holds on a set with $\mu$-measure larger than
$$\frac{4/5\log D - \log D}{4/5\log D-\log m}>9/14.$$
Similarly the second condition of~\eqref{e.stable}
holds on a set with $\mu$-measure larger than
$$\frac{3/4\log D - \log D}{3/4\log D-2\log m+\log D}>9/17.$$
Hence~\eqref{e.stable} holds on a set with $\mu$-measure larger than $9/14+9/17-1>1/6$.}
\end{proof}
\medskip

Let $f\in \diffdiss^{r}(S)$ be a diffeomorphism as in the statement of theorem~\ref{p.strong}
{\color{black} and let $\sigma,\tilde \sigma,\rho,\tilde \rho$ given by proposition~\ref{p1}.
We relax the constants and choose $\tilde \sigma_0<\tilde \sigma$, $\sigma_0>\sigma$,
$\tilde \rho_0<\tilde \rho$, $\rho_0>\rho$ such that
$\frac{\tilde \sigma_0\tilde \rho_0}{\sigma_0\rho_0}>\sigma_0$
still holds.}
For any diffeomorphism $g$ that is $C^1$-close to $f$, one introduces the set
{\color{black} $A_g:=A_g(\sigma_0,\tilde \sigma_0,\rho_0,\tilde \rho_0)$.
By proposition~\ref{p1}, its measure is larger than $1/6$ for any $g$-invariant
ergodic measure which is not supported on a sink.
We also define:}
$$X(g):=\{x\in U\cap A_g,\;
\text{ the branches of $W^s(x)$ are not contained in } f(\overline S)\}.$$

The continuity of the stable manifold obtained in Theorem~\ref{t.stable} gives:
\begin{lema}\label{c2}
Consider a set $\mathcal D^r$ of $C^r$ diffeomorphisms which is bounded for the $C^r$ topology.
Let $x\in X(f)$.
If $g\in \mathcal D^r$ is $C^1$-close to $f$, and if $y\in A_g$ is close enough to $x$,
then $y$ belongs to $X(g)$.
\end{lema}

One can now give the proof of the theorem.

\begin{proof}[Proof of theorem~\ref{p.strong}]
Let $\mathcal D^r$ be a $C^r$-bounded set of $C^r$-diffeomorphisms.
It is enough to check that for any diffeomorphism $g\in \mathcal D^r$
that is $C^1$-close to $f$,
and for any ergodic measure $\nu$ of $g$ which is not a hyperbolic sink,
$\nu(X(g))$ is non-zero.

One can argue by contradiction, consider a sequence of diffeomorphisms $(g_n)$
in $\mathcal D^r$
which converge to $f$ in the $C^1$-topology, and a sequence of ergodic measures $\nu_n$
(not supported on hyperbolic sinks)
converging to an invariant measure $\mu$ of $f$ and assume that
$\nu_n(X(g_n))=0$ for each $n$.
Note that if $\mu$ gives positive measure to a hyperbolic sink, then
$\nu_n$, $n$ large, gives positive measure to the hyperbolic continuation of the sink,
which is a contradiction.
We may now assume that $\mu$-almost every point has one negative Lyapunov exponent
and one non-negative Lyapunov exponent. The same holds for $\nu_n$.
Up to considering a subsequence, one can assume that $(A_{g_n})$
converges for the Hausdorff topology. Note that the limit
is contained in $A_{f}$.

From Proposition~\ref{p1}, there exists a family of compact sets
$Z_n\subset A_{g_n}$ with $\nu_n$-measure larger than $1/6$ which
converge to a compact set $Z\subset A_{f}$, whose $\mu$-measure
is larger or equal to $1/6$.
Since the support of $\mu$ is disjoint from the hyperbolic sinks of $f$,
the set $Z\cap X(f)$ has also measure larger or equal to $1/6$. Hence, up to take slightly smaller subsets,
one can assume furthermore that $Z\subset X(f)$.
Lemma~\ref{c2} (and a compactness argument) implies
that for $n$ large enough, $Z_n$ is contained in $X(g_n)$.
This proves that $\nu_n(X(g_n))>0$ which is a contradiction.
This ends the proof of theorem~\ref{p.strong}.
\end{proof}


\section{Strong dissipation for dynamics close to one-dimensional endomorphisms: proof of Theorem
\ref{t.1D}}
\label{t.1D proof}

We explain in this section how to
build naturally dissipative surface diffeomorphisms from one-dimensional systems
acting on an interval or the circle.

{\color{black}
\subsection{Dynamics close to one-dimensional endomorphisms}
}

\paragraph{Definition of the two-dimensional extension.}
Given a one-dimensional manifold $I$ (the circle $S^1$ or the interval $(0,1)$),
a $C^1$-map $h\colon I\to I$ isotopic to the identity (such that $h(\partial I)\subset \operatorname{Interior}(I)$ in the case of the interval), {\color{black} $\varepsilon>0$ small
and $b\in (-1,1)$ even smaller,}
we get a map $f_b$ on $S:=I\times (-\varepsilon,\varepsilon)$ defined by
\begin{equation}\label{e.deffb}
f_b\colon (x,y)\mapsto (h(x)+y,b(h(x)-x+y)).
\end{equation}
Indeed for any $y\in \RR$ close to $0$ and any $x\in h(I)$,
the sum $x+y$ is well defined and, since $h$ is isotopic to the identity, the difference $h(x)-x$ belongs to $\RR$.

Note that the Jacobian is constant and equal to $b$. When $b\neq 0$,
the map $f_b$ is a diffeomorphism onto its image. When $b=0$
the image $f_0(S)$ is contained in $I\times \{0\}$
and the restriction of $f_0$ coincides with $f\times \{0\}$.

\begin{example}
Let us consider the \emph{quadratic family}
$x\mapsto x^2+c$, for $c\in (-2,-1)$.
It sends the interval
$I=(c/2-1,-c/2+1)$ into its interior.
The map $f_b$ in this case is conjugate to the \emph{H\'enon map}~\eqref{e.henon}
with parameter $a=-b^2/4-c-b/2$ through the map $(x,y)\mapsto (-ax-b/2,-ay-abx)$.
\end{example}

\begin{example}
A class of circle maps isotopic to the identity is the \emph{Arnol'd family} on $S^1$:
\begin{equation}\label{e.arnold}
h_{a,\omega}\colon x\mapsto x+a\sin(2\pi x)+\omega, \quad a,\omega\in \RR.
\end{equation}
\end{example}
\medskip

\begin{figure}
\begin{center}
\includegraphics[scale=0.5]{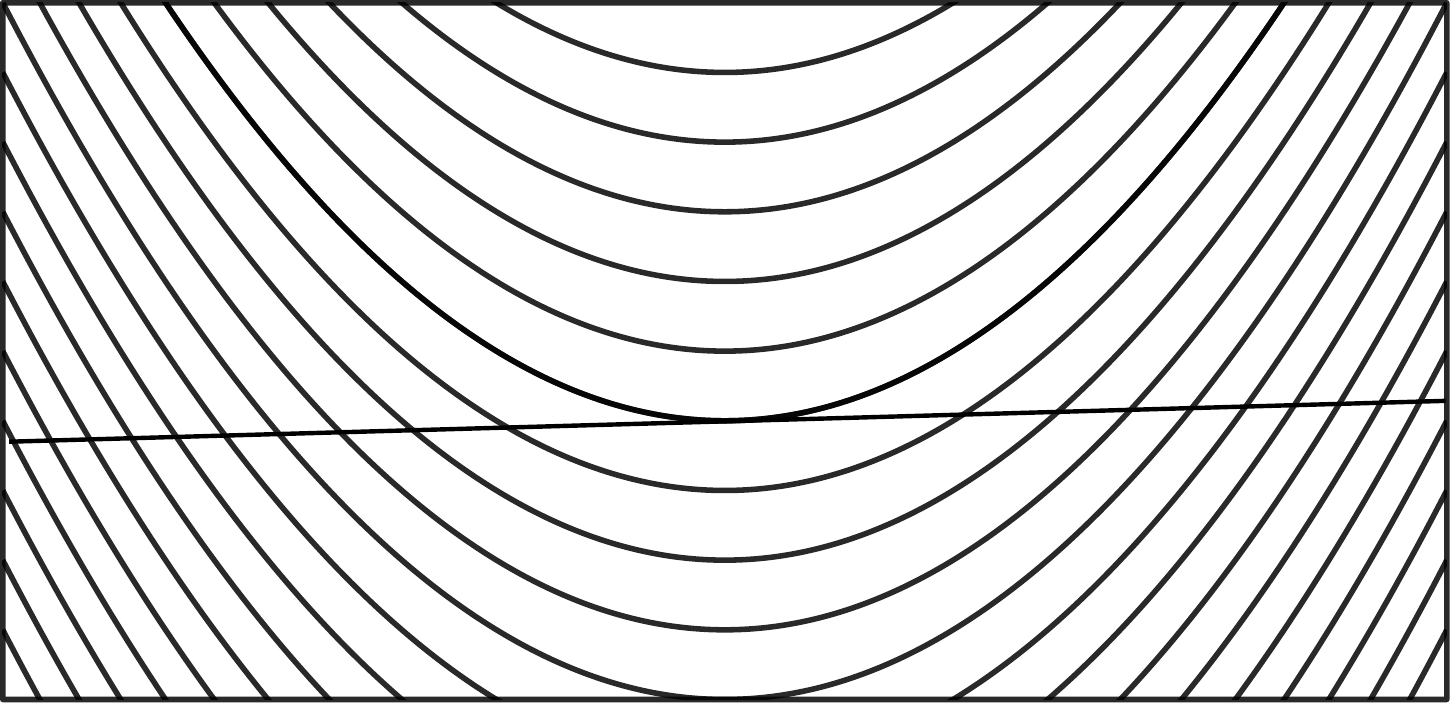}
\end{center}
\caption{
\color{black}
The map $f_0$ associated to the quadratic family $x\mapsto x^2+c$.
Each curve $y=x^2+\mathrm{cte}$ is contracted to a point.
For $b\neq 0$ small, ``most" of the stable manifolds are $C^1$-close to these parabolas.
\label{f.1D}
}
\end{figure}

We can now state the general version of Theorem~\ref{t.1D}.

\setcounter{theo}{\value{theothird}}
\begin{theo}[General version]\label{t.1D}
If $h\colon I\to I$ is a one-dimensional {\color{black} $C^2$}-map, isotopic to the identity such that $h(\partial I)\subset
\operatorname{interior}(I)$ and if $b\in \RR$ is close enough to $0$, then the diffeomorphism
$f_b$ is strongly dissipative.
\end{theo}

\begin{proof}
{\color{black} 
The arguments is close to the proof of the robustness:
for $b=0$, the map $f_0$ is an endomorphism which contracts the curves $h(x)+y=\mathrm{cte}$ to a point:
these curves are analogous to strong stable manifolds. One can check moreover that, for any
ergodic measure which is not supported on a sink, the points in a set with uniform measure
are far from the critical set, implying that these curves cross the domain $I\times (-\varepsilon,\varepsilon)$,
see the figure~\ref{f.1D}.
For $b>0$, the control of the uniformity of the stable manifolds ensures that for points
in a set with uniform measure has stable manifolds close to these parabolas.
}
\medskip

Let $K>1$ such that $\|Dh\|_\infty<K$ and $\|Df_b\|<K$ for any $b$ close to $0$.
{\color{black}
We have $|b|=\|\det Df_b\|_\infty$ and set $m(b)=\|Df^{-1}_b\|^{-1}_\infty$.}
Note that
$$\frac{|b|}{5K}\leq m:=\|Df_b^{-1}\|_\infty^{-1}<  |b|.$$
Let us choose $\delta>0$ small such that
\begin{equation}\label{e.choice-delta}
\frac{2\log K}{2\log K+\frac 1 2 |\log \delta|}<1/15.
\end{equation}

We choose a uniform $L\gg K$ and
{\color{black} we set
$$\sigma(b):=L.|b|,\quad \tilde \sigma(b)=m(b),\quad
\tilde \rho(b):=m^2/|b|,\quad \rho=L^2.|b|.$$
Note that
$\frac{\tilde \rho\tilde \sigma}{\rho\sigma}=m^3/(L.|b|)^3
\geq 1/(5KL)^3$
is larger than $\sigma$ when $|b|$ is small enough.}

We introduce the set $\cA=\cA(f_b)$ of points $x$ having a direction
$E\subset T_xS$ satisfying~\eqref{e.stable}.
Then, the same proof as for proposition~\ref{p1} shows that,
for any ergodic measure ${\color{black} \mu}$ having a non-negative Lyapunov exponent, ${\color{black} \mu}(\cA)> {\color{black} 1/6}$.
In particular $\cA$ is non-empty.

Our goal now is to prove that both branches of the stable manifold of points of $\cA$ intersect $S\setminus f_b(S)$.
For that purpose we recast the proof of theorem~\ref{t.stable}
in order to get a uniformity for the local stable manifold of points in ${\color{black} \cA}$.

\begin{lema}\label{l5} For any $\eta>0$, there are $r_0,b_0>0$ such that any map $f_b$, $|b|<b_0$ has the following property:
at any $x_0\in \cA$,
the $r_0$-neighborhood of $x_0$ in $W^s(x_0)$ is a disc of radius $2r_0$ which is
$\eta$-close to a linear segment:
all the tangent spaces are $\eta$-close to a fixed direction.

\end{lema}

\begin{proof} One reproduces the constructions made at section~\ref{ss.stable},
keeping the dependence in $|b|$.
{\color{black} We set $\lambda_1:= 2\sigma$ and $\lambda_2:=\tilde \rho /2$.}
Note in particular that
$C_0$ in~\eqref{e.bound-C_0} is uniform in $b$.

As before, one defines a local diffeomorphism $h_n=\Delta_{n+1}\circ g_n\circ \Delta_n^{-1}$
from a neighborhood of $0\in T_{f^n(x)}S$ to a neighborhood of $0\in T_{f^{n+1}(x)}S$.
Its tangent map $H_n$ has an inverse of the form
$$H_n^{-1}=\begin{pmatrix} \frac 1 a & \frac {-d}{ac}\\ 0 & \frac 1 c\end{pmatrix}=\frac 1 a \begin{pmatrix} 1 & \frac {-d} c\\ 0 &\frac a c \end{pmatrix}.$$
From~\eqref{e.c} and~\eqref{e.d}, we have
$$\bigg| \frac a c \bigg|\leq {\color{black} \lambda_2=\frac {m^2}{2|b|}\leq |b|,}\quad
\text{ and } \bigg| \frac d c \bigg|\leq
{\color{black} \|Df_b^{-1}\|.\|Df_b\|\lambda_2\leq K.}$$
{\color{black} Let us consider the horizontal cone $\mathcal{C}$ of size $\eta$:
$$\mathcal{C}=\{(u_1,u_2),\; \eta.|u_1|\geq |u_2|\}.$$
One can reduce $\eta>0$ and assume that $\eta K\ll 1$.
Hence the} horizontal cone is preserved
by any linear map which is $\frac \eta {2a}$-close to $H_n^{-1}$.
The $C^2$-norm of $g_n^{-1}$ is bounded by $C.|b|^{-2}$, where $C$ is uniform in $b$.
This shows that $Dh_n^{-1}$ is $\frac \eta {2a}$-close to $H_n^{-1}$ on a ball of radius
$$r_{n+1}=\frac{\eta |b|^2}{2C.a.\|\Delta_n\|.\|\Delta_{n+1}^{-1}\|^2}.$$
Since $\|\Delta_{n+1}^{-1}\|<1$ and $\|\Delta_n\|=A_nB_n$, the estimates~\eqref{e.a}, \eqref{e.uniformA2}, \eqref{e.uniformB} imply
$$r_{n+1}\geq
{\color{black} \frac{\eta |b|^2}{2CC_0^2\lambda_1}.\bigg(\frac {\lambda_2\tilde \sigma}{\lambda_1\rho}\bigg)^{n}\geq 
\frac{\eta|b|}{4CC_0^2L}.(500L^3K^3)^{-n}.}$$
The maps $Dh_n^{-1}$ expand vectors in the horizontal cone
by more than ${\color{black} \lambda^{-1}_1=1/(2L.|b|)}$ from~\eqref{e.a}.
One deduces that there exists a Lipschitz graph containing $0$ in $T_{x}S$ with
uniform size
{\color{black} 
$$r_0\geq \frac{\eta}{8CC_0^2L^2}$$}
 whose iterates by the
sequence of local diffeomorphisms $h_n$ remain Lipschitz for each $n$,
and after $n$ iterates have radius smaller than $r_n$.

Projecting in $S$, by $\exp\circ \Delta_0$, one gets a local stable manifold at $x$
with uniform size $2r_0$ and tangent to the projection of the constant horizontal cone
which is $\eta$-thin.
\end{proof}
\medskip

It remains to control the slope of the local stable manifolds.
This is done outside a neighborhood of a ``critical region".
Let us define
$$\mathcal{C}=\{x, |Dh(x)|\leq \delta\}\times (-\varepsilon,\varepsilon).$$
\begin{lema}\label{l6}
{\color{black} For $|b|$ small enough, and any ergodic measure $\mu$
of $f_b$ having one Lyapunov exponent non-negative,} $\mu(\cA\setminus \mathcal{C})>1/10$.
\end{lema}
\begin{proof}
{\color{black} Since $\mu(\cA)>1/6$,}
it is enough to check that $\mu(\mathcal{C})<{\color{black} 1/15}$.
Note that for $|b|$ small,
the Lyapunov exponent of $\mu$ is bounded above by
{\color{black} $\mu(\mathcal{C}) 2\log(\delta)+(1-\mu(\mathcal{C}))2\log\|Dh\|_\infty$.}
Since $\mu$ does not charge sinks, the exponent is non-negative so that
{\color{black} with~\eqref{e.choice-delta} one gets:}
$$\mu(\mathcal{C})\leq {\color{black} \frac{2\log K }{2\log K+\frac 1 2 |\log \delta|}<1/15.}$$
\end{proof}

\begin{lema}\label{l7}
For points $z$ in $\cA\setminus \mathcal{C}$,
the slope of the direction $E$ is larger than $\frac 3 4 \delta$.
\end{lema}
\begin{proof}
Let $z=(x,y)$, let $v$ be a unit vector tangent to $E$ and let $\alpha$ be the angle
between $v$ and the line $\RR\times \{0\}$.
From~\eqref{e.stable}, its image has norm smaller or equal to $\sigma$.
From the definition of $f_b$, the first projection of $Df_b.v$ has a modulus larger than
$\|Dh(x)\||\cos(\alpha)|-|\sin(\alpha)|$.
Since $z$ does not intersect $\mathcal{C}$,
the modulus $\|Dh(x)\|$ is larger than $\delta$.
One thus gets
$|\tan(\alpha)|\geq \delta - \sigma/|\cos(\alpha)|$
which implies that the slope is larger than $3\delta/4$
when $\sigma=|b|^{4/5}$ is small enough.
\end{proof}

\paragraph{\it End of the proof of theorem \ref{t.1D}.}
{\color{black} From lemmas~\ref{l5} and~\ref{l7}, if $|b|$ is small enough,}
for points in $\cA\setminus \mathcal{C}$ the stable manifold has uniform size
and slope larger than $\delta/2$.
Consequently, if $\varepsilon$ has been chosen small enough, this proves that for $|b|$ small and
for points in $\cA\setminus \mathcal{C}$, both branches of the stable manifold
intersect the boundary of $I\times (-\varepsilon,\varepsilon)$.
{\color{black}
By lemma~\ref{l6}, this set has positive measure for any
ergodic $f_b$-invariant measure $\mu$ having a non-negative Lyapunov exponent.
This proves that $f_b$ is strongly dissipative.}
\end{proof}

{\color{black}
\subsection{H\'enon diffeomorphisms}
Let $f$ be a dissipative H\'enon automorphism of $\CC^2$ with jacobian $b\in \CC$ (with $0<|b|<1$)
and degree $d:=\deg(f)\geq 2$.
In~\cite{DL} it is proved that if $|b|<d^{-2}$, then
for any periodic point $p$ which is not a sink, the connected component of $p$ in $W^{ss}_{\CC}(p)\cap K^-$ is $\{p\}$.
In that setting $W^{ss}_{\CC}(p)$ is the strong stable manifold in $\CC^2$ (a biholomorphic copy of $\CC$),
and $K^-$ is the set of points $z\in \CC^2$ whose backward orbit remains bounded. We explain here
how the proof of this result also gives the version of Theorem~\ref{t.1D} stated in the introduction.
For more details, the reader should consult~\cite{DL}.

Let us consider an ergodic measure $\mu$ for the action of $f$ in $\CC^2$,
which is not a sink. It has a unique negative Lyapunov exponent $\lambda^-(\mu)\leq \log(b)$ and $\mu$-almost every point $x$
has a strong stable manifold $W^{ss}_{\CC}(x)$: it is the image of $\CC$ by a holomorphic parametrization
$\phi_x\colon \CC\to \CC^2$.

Let $G^-\colon \CC^2\to \RR$ denotes the \emph{backward Green function} defined by
$$G^-(z)=\lim_{n\to +\infty} \frac 1 {d^n} \log^+\|f^{-n}(z)\|.$$
For $\mu$-almost every $x$, one gets a non-negative subharmonic function $g^-_x:=G^-\circ \phi_x$ which vanishes exactly no the set $K^-$.
We consider its \emph{order}
$$\rho(g^-_x):=\limsup_{r\to \infty} \frac 1 r \log \log \|g_x^-\|_{B_r},$$
where $B_r$ denotes the ball of radius $1$ centered at $0$.
From~\cite[Section 4]{BS}, it is equal to
$$\rho(g^-_x)=\frac{\log(d)}{|\lambda^-(\mu)|}.$$
Under the assumption $|b|<d^{-2}$, the order is thus smaller than $1/2$.
Since the order is non-zero it is not bounded.
Since the supremum of two subharmonic functions is subharmonic,
one can consider for each $T>0$, the subharmonic function
$x\mapsto \max(g(x), T)-T$
and apply the following version of Wiman's theorem (see~\cite[Theorem 35]{LP}):

\begin{teo*}[Wiman] Let $g\colon \CC\to \RR$ be a non-constant subharmonic function with order smaller than $1/2$.
Then all components of $g = 0$ are bounded.
\end{teo*}

Hence for $R>0$,
the connected component of $\{z, \varphi_x(z)\leq R\}$ containing $x$ is bounded.

\begin{proof}[Proof of Theorem~\ref{t.1D} for H\'enon maps]
Let us consider a real H\'enon diff\'eomorphism $H_{a,b}$ as defined in the introduction, such that $0<|b|<1/4$.
Its degree equal $d=2$, so that the condition $|b|<d^{-2}$ holds.
Let $\mu$ be any ergodic probability measure on $\RR^2$ which is not supported on a sink.
For $\mu$-almost every point $x$, the strong stable manifold $W^s(x)$ is the restriction of $W^s_{\CC}(x)$
to $\RR^2$.
Moreover, from Pesin theory there exists $\varepsilon>0$ such that
for any $N\geq 1$,
the set of point $y\in W^s(x)$ satisfying
$d(f^k(x),f^k(y))<\varepsilon$ for any $k\geq N$ is a compact subset
for the intrinsic topology of the curve $W^s(x)$.
In particular if one considers a parametrization $\psi_x\colon \RR\to W^s(x)$,
its lift by the parametrization of $W^s_{\CC}(x)$ is proper, i.e.
we have $|\varphi_x^{-1}\circ \psi_x(t)|\to \infty$ as $|t|\to \infty$.
From the previous arguments, one deduces that the connected sets
$\psi_x[0,+\infty)$ and $\psi_x[0,-\infty)$ are not bounded in $\RR^2\subset \CC^2$.
\end{proof}
}

\begin{figure}[ht]
\begin{center}
\includegraphics[scale=0.4]{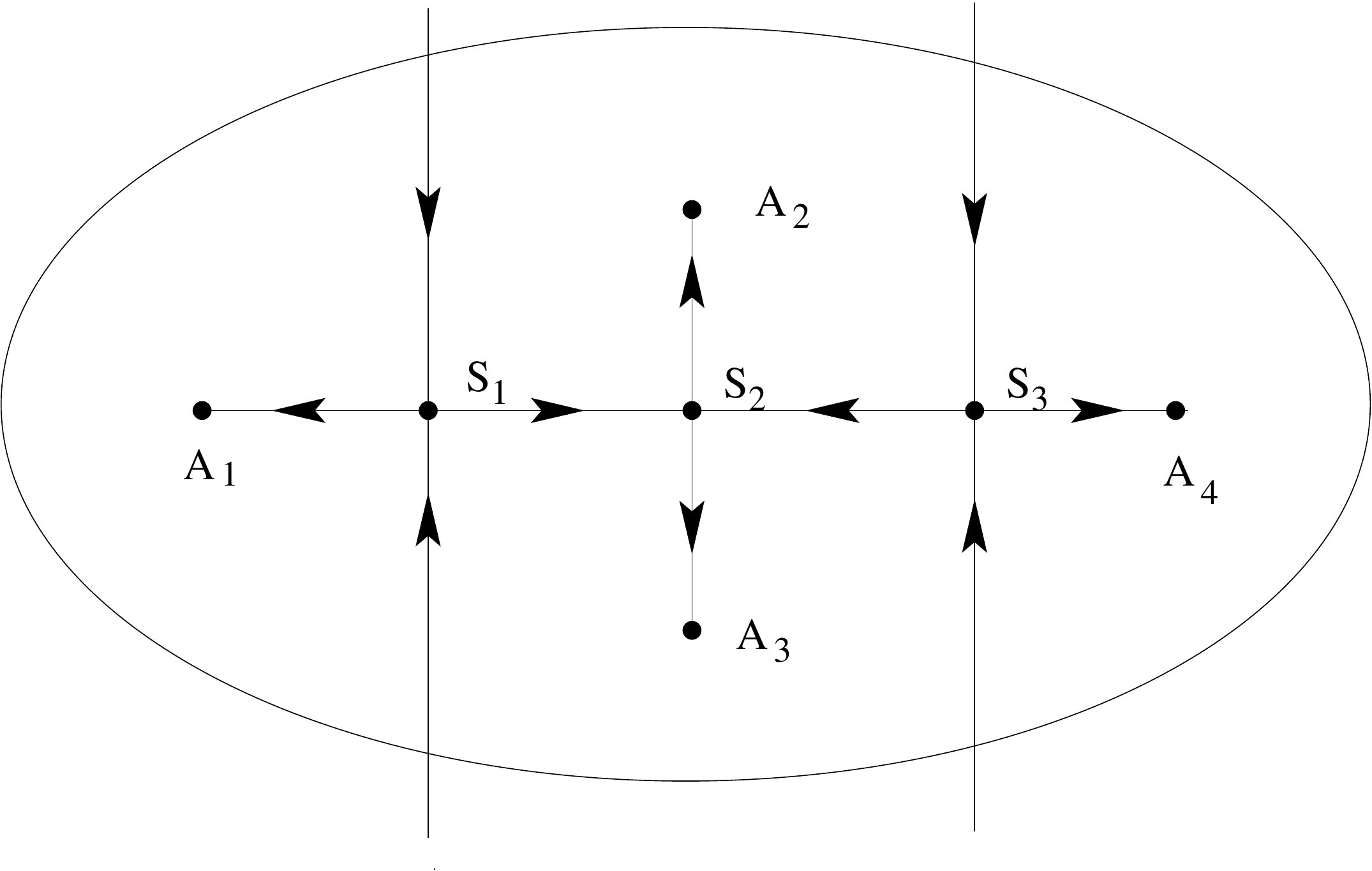}
\end{center}
\caption{ Dissipative diffeomorphism that is not strongly dissipative
(the stable manifold of $S_2$ is contained in a compact subset of the open disc $\DD$ and does not separates).
\label{f.example}}
\end{figure}

{\color{black}
\subsection{A diffeomorphism which is not strongly dissipative}\label{ss.counter-example}
}

We close the section by showing a dissipative diffeomorphisms which is not strongly dissipative. The attracting domain is a disk  and the non-wandering domain is given by four attracting fixed points (denoted as $A_i$ in figure~\ref{f.example} and three saddle fixed points (denoted as $S_i$). Both branches of the stable manifolds of the saddles $S_1$ and $S_3$ are not contained in the disk but both branches of the saddle $S_2$ are in the disk; one branch coincides with one unstable branch of $S_1$ and the other with one unstable of $S_3.$ To make the example dissipative, it is required to chose appropriately the eigenvalues of the saddles.
\medskip


\section{Reduction to one-dimensional dynamics: proof of theorem~\ref{reduction}}
\label{sec proof reduction}

{\color{black} It is well known that the space of leaves for foliations in the plane generates a one-dimensional structure.
In our setting the strong dissipation provides us with a large collection of disjoint curves: the stable manifolds.
The idea of the proof of theorem~\ref{reduction} is to ``quotient" the disc along these curves.}
\medskip

\begin{figure}[ht]
\begin{center}
\includegraphics[scale=0.4]{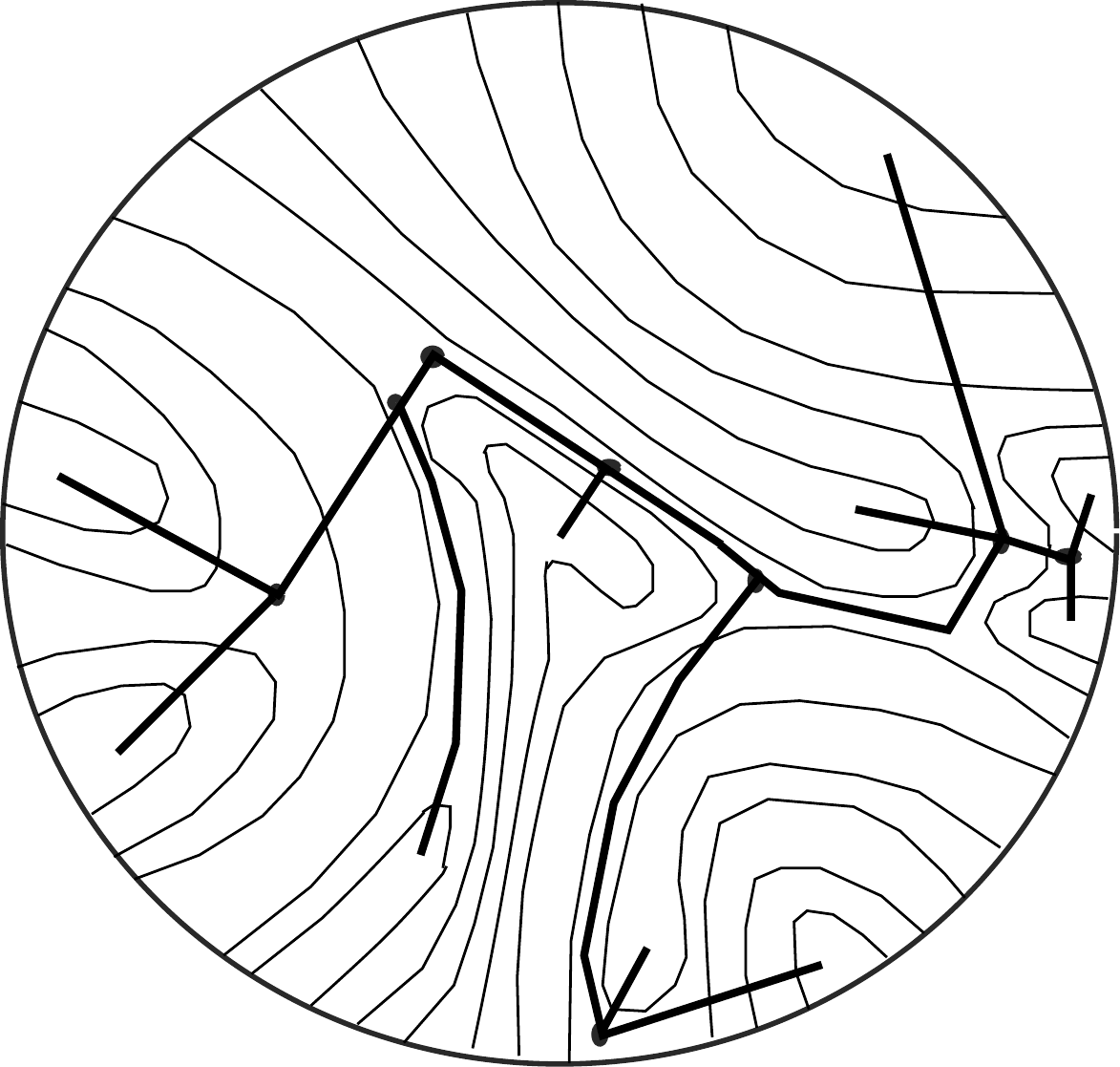}
\end{center}
\caption{ The one-dimensional structure associated to the family of stable manifolds.}
\end{figure}

One chooses a countable collection $\Gamma$ of proper $C^1$-arcs  in $\mathbb{D}$ with the following properties.
\begin{enumerate}
\item Each $\gamma\in \Gamma$ is contained in the stable manifold
$W^{s}(x)$ of a regular point $x$ for an aperiodic ergodic measure
(but $x$ is not necessarily in $\gamma$). In particular, elements of $\Gamma$ are pairwise disjoint or coincide.
\item\label{i.2} Each $\gamma\in \Gamma$ is the $C^1$-limit of arcs in $\Gamma$ and is accumulated on both sides.
\item\label{i.3} For $\gamma\in \Gamma$, the connected components of $f^{-1}(\gamma)\cap \mathbb{D}$
which intersect $f(\mathbb{D})$ are also in $\Gamma$.
\item\label{i.4} For each aperiodic ergodic measure $\mu$,
there exists a full measure set of points $x$ such that the connected components
of $W^{s}(x)\cap \mathbb{D}$ are $C^1$-limits of arcs in $\Gamma$ and are accumulated on both sides.
\end{enumerate}

\paragraph{The space $X$.}
One considers open connected surfaces $s$ of $\mathbb{D}$ bounded by finitely many elements of $\Gamma$.
One denotes $\Sigma$ the collection of sequences $(s_n)$ of such surfaces such that
$\operatorname{Closure}(s_{n+1})\subset s_n$ for each $n$
and one sets $(s_n)\leq (s'_n)$ if for any $n$, there is $m$ such that
$\operatorname{Closure}(s_{m})\subset s'_n$.
Let $\Sigma_0$ be the collection of sequences that are minimal for the relation $\leq$.
One defines $X$ as the quotient of $\Sigma_0$ by the relation ``$(s_n)\leq (s'_n)$ and $(s'_n)\leq (s_n)$".
Note that $\Gamma$ may be identified to a subset of $X$:
the arc $\gamma$ is represented by sequences $(s_n)$ such that $\gamma\subset s_n$ for each $n$
and $\cap s_n=\gamma$; such a sequence exists from our assumption~\ref{i.2} on $\Gamma$.

One defines a topology on $X$. A (countable) basis is defined
by considering all the sequences $(s_n)$ such that the closure of $s_n$ is contained in a given surface $s$
for $n$ large enough. Note that $X$ is separated.
Moreover it is regular: any non-empty closed set $C$ and any point $x$ in the complement can be separated.
Indeed, since $C$ is closed, $x$ has a neighborhood, defined by a surface $s$, disjoint from $C$. By definition of $\Gamma$, one can build
two surfaces $s\supset s_1\supset s_2\supset x$. The surface $s_2$ defines a neighborhood of $x$.
The complement of the closure of $s_1$ defines a neighborhood of $C$. Both are disjoint.

\begin{lema}
$X$ is a real tree.
\end{lema}
\begin{proof}
The space $X$ is metrizable, by Urysohn theorem (a separated, countable basis, regular topological space is metrizable).
\medskip

The space $X$ is compact. Indeed, let $(x_k)$ be a sequence in $X$
and consider the collection $(\gamma_n)$ of arcs of $\Gamma$.
One can assume, for each $n$ that the sequence $(x_k)$ is different from $\gamma_n$
for each $k$ large, hence (up to take a subsequence) is contained in a same component
of the complement of $\gamma_n$. One defines in this way a minimal decreasing sequence $s_n$
of surfaces and all the $x_k$, $k$ large, belong to $s_n$.
Thus the sequence $(x_k)$ converges to the point defined by $(s_n)$.
\medskip

One say that $x$ separates $y$ and $z$ if for any small neighborhood of $x$ (defined by a surface $s$),
$y$ and $z$ are contained in different components of the complement of $s$.
Let us index the elements of $\Gamma$ and consider the subfamily $(\gamma_n)$ of those
that separate $y$ and $z$. They are ordered by the separation property.
The maximum of the distance between two consecutive $\gamma_n$, $n\leq n_0$, goes to zero as
$n_0$ goes to $\infty$. Otherwise, one gets two sequences $\gamma_{n_k}$, $\gamma_{m_k}$,
the first increasing, the second decreasing, for the order between points that separate $y$ and $z$.
Moreover the distance $d( \gamma_{n_k},\gamma_{m_k})$ does not converge to $0$.
These sequences converge to two different points $a,b$ that separate $y$ to $z$.
Since $a,b$ are different, there should exist $\gamma_n$, $n$ large, that separate $a$ to $b$,
contradicting the construction of $a$ and $b$.
This proves that the union of $\{y,z\}$ with the set of points which separate $y$ and $z$
is homeomorphic to the interval $[0,1]$. Thus $X$ is path connected.
Note that any path joining $y$ to $z$ should contain this set of point (since they separate).
This shows that $X$ is a real tree.
\end{proof}

\paragraph{The maps $\pi,h$.} Any point $\gamma\in \Gamma$ naturally projects to $X$.
For each point $x\in \mathbb{D}\setminus \bigcup\{\gamma\in \Gamma\}$, one can consider
the component $s_n$ of $\mathbb{D}\setminus \{\gamma_1\cup\dots\cup \gamma_n\}$ which contains $x$.
One gets in this way a minimal decreasing sequence which defines a point $\pi(x)\in X$.
This map $\mathbb{D}\to X$ is obviously continuous.

The map $f$ induces a continuous map $h$ on $X$, defined as follow.
Let $(s_n)$ be a minimal decreasing sequence and consider
the intersection $C$. Any two points $y,z$ in $C$ have their image by $f$ which are not separated
by a curve $\gamma$. Otherwise, this curve $\gamma$ crosses $f(\mathbb{D})$,
the pre image $f^{-1}(\gamma)$ has a connected component separating
$y,z$, which belongs to $\Gamma$ from the assumption~\ref{i.3}; this proves that the family $\Gamma$ separates $y$ and $z$, a contradiction.
The image of the point $x=\pi(C)$ by $h$ is $\pi(f(C))$.
\medskip

Note that points that are regular for different aperiodic ergodic measures
belong to different stable arcs, hence are mapped to different points in $X$ by the assumption~\ref{i.4}.
Thus different aperiodic ergodic measures project to different measures on $X$.
\medskip

{\color{black} We now bound the topological entropy $h_{top}(f)$ of $f$ and $h$. We fix $\alpha>0$.
By the variational principle, there exists an ergodic measure $\mu$ for $f$
whose entropy satisfies $h_\mu(h)>h_{top}(f)-\alpha$.
If $\mu$ is supported on a periodic measure, we have $h_{top}(f)<\alpha$,
otherwise, $\mu$ is aperiodic.
Let $\nu=\pi_*(\mu)$. By~\cite{LW}, the entropies of $\mu$ and $\nu$ can be compared:
for each $x\in X$, one considers the topological entropy $h_{top}(f,Z)$ of the preimage
$\pi^{-1}(x)$; one then have:
$$h_{\mu}(f)=h_{\nu}(h)+\int h_{top}(f,\pi^{-1}(x))d\nu(x).$$

For $\nu$-almost every point $x$, the preimage $\pi^{-1}(x)$ is contained in the stable set of $x$.
In particular for any $\varepsilon>0$, the orbits of points in $\pi^{-1}(x)$ do not $\varepsilon$-separate
after some time: the topological entropy $h_{top}(f,\pi^{-1}(x))$ is equal to zero.
We have thus obtained $h_{top}(f)<h_\mu(f)+\alpha\leq h_\nu(h)+\alpha\leq h_{top}(h)+\alpha$
and the topological entropies of $f$ and $h$ coincide.}

The proof of theorem~\ref{reduction} is now complete. \qed


\section{Density of periodic points: proof of theorem \ref{t.measure revisited}}
\label{sec measures}

\paragraph{One dimensional sketch.}
\label{one dim sketch}

First  we prove of theorem~\ref{t.measure revisited} in the context of one-dimensional dynamics
(our statement is slightly more general than~\cite[theorem 1]{LSY}). 

\begin{propo}\label{density one dim}
For any continuous map $f\colon[0,1]\to [0,1]$, the set of periodic points is dense in the set of recurrent points.
\end{propo}
\begin{proof}
Let us consider a recurrent point $x_0$ and a nearby iterate $x_1=f^n(x_0)$. Without loss of generality, one can assume
that $x_0<x_1$.
We claim that there is a periodic  point for $g:=f^n$ inside $(x_0, x_1)$.
Note that $x_0$ is still recurrent for $g$. We take the first positive integer $k$ such that $g^k(x_1) < x_1$  
Such an integer exists since $x_0$ is recurrent and $x_0< x_1$.  By the choice of $k$ observe that $g^k(x_0)= g^{k-1}(x_1) \geq x_1$. Therefore we have a continuous map $g^k\colon [x_0, x_1]\to [0,1]$ such that $g^k(x_0)\geq x_1$ and $g^k(x_1)< x_1$. Hence the graph of $g^k$ crosses the diagonal in a point in $(x_0,x_1)$ and so there is a fixed point in $(x_0, x_1).$ 

We recast the previous proof, avoiding an explicit use of the order, and in such a way that it can easily be generalized for strongly dissipative diffeomorphisms in the disk. 

Under the same choices of $x_0$ and $x_1$ as above, we define the intervals $D^-=[0, x_0]$ and $D^+=[x_1, 1]$ and we take the first positive integer $k$ such that $g^k(x_1)\notin D^+.$
Such an integer exists since $x_0$ is recurrent and does not belong to $D^+$.  By the choice of $k$ observe that $g^k(x_0)= g^{k-1}(x_1) \geq x_0$
Let $h\colon [0, 1]\to [x_0, x_1]$ be a continuous map which coincides with the identity on $[x_0,x_1]$
and such that $h([0,x_0])=x_0$, $h([x_1,1])=x_1$.
Then the map $h\circ g^k:[x_0, x_1]\to [x_0,x_1]$ has a fixed point $p\in [x_0, x_1].$
Note that $h\circ g^k(x_0)=x_1$ (since $g^k(x_0)\in D^+$) and $h\circ g^k(x_1)\neq x_1$ (since $g^k(x_1)\not\in D^+$).
Therefore $p$ belongs to $(x_0, x_1)$. Since $h$ is the identity on $(x_0, x_1)$, the point $p$ is a fixed point of $g^k$.
\end{proof}

As a minor detail, observe that in the  previous proof it is enough to choose $k$ such that $g^k(x_1) < x_1$ and $g^{k-1}(x_1) > x_1$.

The key factor in the proof of proposition \ref{density one dim} is a simple one: {\em  a point disconnect an interval}. 
The last observation does not have an immediate two-dimensional counterpart (points does not separate a two dimensional domain).
However, local stable manifolds, under the assumption of strong dissipativeness, do separate a  two dimensional domain.
\bigskip

\begin{proof}[\bf Proof of theorem~\ref{t.measure revisited}]
Let us recall first some classical result and definitions about invariant measure and Pesin theory.

\begin{lema}\label{localst}
Let $f$ be $C^r$-diffeomorphism, $r>1$ of a surface $S$.
Let $\mu$ be an ergodic invariant measure with compact support having two Lyapunov exponents
$\lambda^-<0\leq \lambda^+$.

Then there exist $C,\lambda>0$, a compact set $B$ with positive $\mu$-measure
(called \emph{Pesin block}) and a continuous family of $C^1$-embeddings
$(\phi_x)_{x\in B}$ of $[-1,1]$ into $S$
such that for $x\in B$,
\begin{itemize}
\item[--] $\phi_x(0)=x$ and the stable manifold $W^s(x)$ contains the image of $\phi_x$,
\item[--] for any $n\geq 0$, the length of $f^n(\phi_x([-1,1]))$ is smaller than $C e^{-\lambda.n}.$
\end{itemize}
\end{lema}
\begin{proof}
From~\cite{Pe}, there exists a measurable set $X\subset S$ with positive measure and
a measurable family of $C^1$-embeddings $(\varphi_x)_{x\in X}$ such that the image of $\varphi_x$ is a local
manifold of $x$. Lusin's theorem allows to find a compact set $B\subset X$ as required.
\end{proof}

To prove the theorem, since any regular point has an iterate in a hyperbolic block, it is enough to show that arbitrary close to any point $x\in B$ there is a periodic point. By Poincar\'e recurrence theorem, one can assume that any point in $B$ is positively recurrent.

Since $f$ is strongly dissipative, both branches of $W^s(x_0)\setminus \{x_0\}$ meets $S\setminus f(S)$.
For any $\delta>0$, it is thus possible to find a forward iterate $f^n(S)$ which intersect both branches at two points
$\delta/2$-close to $x$ in $W^s(x)$. Moreover one can modify the boundary of $f^n(S)$ in order
to get a $C^1$-loop $\gamma$ transverse to $W^s(x)$. The disc $D\subset S$ bounded by $\gamma$
still satisfies $f(\overline D)\subset \operatorname{Interior}(D)$.

Let $n\geq 1$ be a large integer such that $f^n(x)$ belongs to $B$ and is close to $x$.
Then by Lemma~\ref{localst}, the local manifold of $f^n(x)$ is $C^1$-close to the local manifold of $x$,
and in particular is also transverse to $\gamma$.
One deduces that the connected components of $W^s(x)\cap D$ and $W^s(f^n(x))\cap D$
containing respectively $x,f^n(x)$, together with small arcs in $\gamma$ enclose a region $R$ diffeomorphic
to the square. If $m$ has been chosen large enough, the diameter of $R$ is smaller than $\delta$.
Moreover $D\setminus R$ has two connected components, whose closure are two topological discs, denoted by $D^+,D^-$:
for instance one can choose $x\in D^-$ and $f^n(x)\in D^+$.
By construction $D^-\cap R$ (resp. $D^+\cap R$)  is contained in the local stable manifold of $x_0:=x$ (resp. $x_1:=f^n(x)$).
See the figure.

\begin{figure}[ht]
\begin{center}
\includegraphics[scale=0.6]{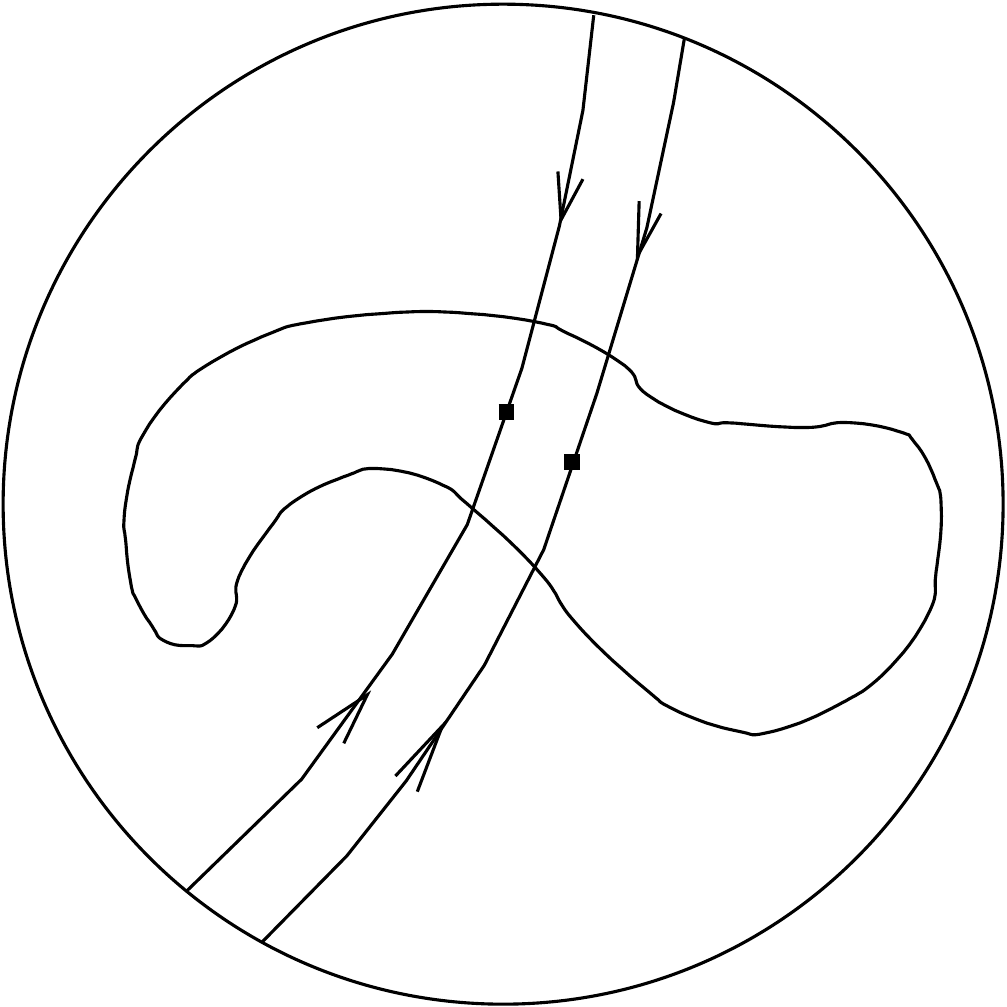}
\begin{picture}(0,0)
\put(-147,91){$D^-$}
\put(-60,67){$D^+$}
\put(-75,91){$\tiny x_1$}
\put(-111,102){$\tiny x_0$}
\put(-90,97){$\tiny R$}
\put(-187,135){$\DD$}
\end{picture}
\end{center}\end{figure}

Let $g=f^k$ and let $k$ be the smallest positive integer such that $g^k(x_0)\in D^+$ and $g^{k+1}(x_0)\notin D^+$.
Such an integer exists since $x_0$ is recurrent for $g$ and since $x_1=g(x_0)\in D^+.$ 
Similarly as in the one dimensional case, we consider a continuous map $h\colon D\to R$
such that the restriction of $h$ to $R$ is the identity,
$h(D^-)= D^-\cap R$ , $h(D^+)= D^+\cap R$.
In particular, $h\circ g^k$ sends $R$ into itself and therefore has a fixed point $p$ in $R$.

Since $f^k(x_0)\in D^+$ and since $D^-\cap R$ is the local stable manifold of $x_0$, its image meets $D^+$ (and is contained in $D$);
since $D^+\cap R$ is also a stable manifold, either it contains or it is disjoint from $f^k(D^-\cap R)$.
Consequently $f^k(D^-\cap R)\subset D^+$.
Similarly, $f^k(x_1)\notin D^+$, hence $f^k(D^-\cap R)\cap D^+=\emptyset$.

One deduces that $p$ does not belong to $R\cap (D^+\cup D^-)$: by definition of $h$ it has a unique pre image and $h^{-1}(p)=p$.
This implies $g^k(p)=p$. Hence $f$ has a periodic point arbitrarily close to $x$, as required.
\end{proof}
\bigskip

\begin{proof}[\bf Proof of corollary~\ref{c.henon}]
Let us fix $a\in (1,2)$ and define $S=\{(x,y):\; |x|<1/2+1/a, \; |y|<1/2-a/4\}$ and $C:=\{(x,y), |x|>|y|\text{ and } |x|>3\}$.
Note that the following properties hold:
\begin{itemize}
\item[--] $f(C)\subset C$ and any forward orbit $(x_n,y_n):=H_{a,b}^n(x,y)$ in $C$ satisfies $|x_n|\underset{n\to +\infty}\longrightarrow \infty$.
\item[--] Any point $(x,y)$ has a positive iterate in $S$ or in $C$.
\item[--] The quadratic map $x\mapsto 1-ax^2$ has a unique fixed point in $(-3,-1/2-1/a)$
and any other orbit escapes this interval.
Consequently, if one fixes $\varepsilon>0$ small enough,
for any $b$ close to $0$, the map $H_{a,b}$ has a unique fixed point $p$ in
$(-3,-1/2-1/a)\times (-\varepsilon, \varepsilon)$.
Moreover if $b$ is close enough to $0$, any point in
$(-3,-1/2-1/a)\times (-\varepsilon, \varepsilon)$ has a forward iterate in $S$ or
in $C$.
\item[--] If $b$ is chosen close to $0$,
the region $\{|x|<3,|y|<3\}$ is mapped by $H_{a,b}$ in the union of $C$ with
$(-3,-1/2-1/a)\times (-\varepsilon, \varepsilon)$.
\end{itemize}
One deduces:
\begin{lema}
For any $a\in (1,2)$, if $|b|>0$ is close enough to $0$,
the H\'enon map has a unique fixed point $p$ in $(-3,-1/2-1/a)\times (-1, 1)$
and any forward orbit satisfies one of the following properties:
\begin{itemize}
\item[--] it escape to infinity (the orbit has no accumulation point in the plane),
\item[--] it converges to $p$,
\item[--] it is attracted by the trapping region $S=\{(x,y):\; |x|<1/2+1/a, \; |y|<1/2-a/4\}$.
\end{itemize}
\end{lema}

Since the dynamics of $H_{a,b}$ is strongly dissipative in $S$
(theorem~\ref{t.1D-bis}), the theorem~\ref{t.measure revisited} may be applied in this region
and concludes the proof.
\end{proof}


\bigskip

\begin{tabular}{l l l}
\emph{Sylvain Crovisier}
& \quad\quad \quad &
\emph{Enrique Pujals}
\medskip\\

Laboratoire de Math\'ematiques d'Orsay
&& IMPA\\
CNRS - UMR 8628
&& Estrada Dona Castorina 110,\\
Universit\'e Paris-Sud 11
&&  22460-320 Rio de Janeiro, Brazil.\\
Orsay 91405, France.
&&
\end{tabular}

\end{document}